\documentclass[12pt]{elsarticle}

\usepackage{amssymb}

\usepackage{amsmath}
\usepackage{graphicx}
\usepackage{verbatim}
\usepackage{amssymb}
\usepackage{graphicx}
\usepackage{rotating}
\usepackage{lscape}
\usepackage{subfig}
\usepackage{algorithm}
\usepackage[noend]{algpseudocode}
\usepackage{amsthm}
\newtheorem{theorem}{Theorem}
\newtheorem{lemma}[theorem]{Lemma}

\usepackage{graphicx}

\journal{Springer}

\begin{document}

\begin{frontmatter}



\title{Variable Reduction For Quadratic Unconstrained Binary Optimization}

\author[label1]{Amit Verma}
\author{Mark Lewis}
\address[label1]{Craig School of Business, Missouri Western State University, Saint Joseph, MO, 64507, United States}

\begin{abstract}
Quadratic Unconstrained Binary Optimization models are useful for solving a diverse range of optimization problems.  Constraints can be added by incorporating quadratic penalty terms into the objective, often with the introduction of slack variables needed for conversion of inequalities. This transformation can lead to a significant increase in the size and density of the problem. Herein, we propose an efficient approach for recasting inequality constraints that reduces the number of linear and quadratic variables. Experimental results illustrate the efficacy.  

\end{abstract}


\end{frontmatter}


\section{Introduction}

Quadratic Unconstrained Binary Optimization (QUBO) has emerged as a popular framework for modeling a wide variety of combinatorial optimization problems (see \cite{kochenberger2014unconstrained, glover2019quantum} for details).  A QUBO model is of the form $min \; {x' Q x}$ where ${x}$ represents the binary decision vector. By definition, a QUBO model is quadratic and does not involve any constraints. However, any equality constraint of the form ${A x} = b$ can be transformed into a penalty term $M ({A x} - b)^2$ which is added to the objective function. Here $M$ represents a large positive number that penalizes the deviation of the left-hand side from the right. We can integrate a linear inequality constraint into the objective by using additional slack variables. First, assuming that the coefficients of the linear inequality are non-negative integers, we transform an inequality constraint ${A x} \le b$ to $Ax \in \{0,1,...,b-1,b\}$. If any coefficient $a_i$ of a variable $x_i$ is negative, we use the complement variable $y_i = 1 - x_i$ to reverse the sign of the coefficient. Next, we interpret the inequality constraint as an equality constraint ${A x} = {D s}$ where $\mathbf{s}$ represents the slack variables and ${D}$ represents the coefficients of the slack variables which are chosen such that the entire slack range $\{0,1,...,b-1,b\}$ is covered. Note without loss of generality that all the involved coefficients of $Q,A$ and $D$ matrices are typically integers. Next, the resulting linear equality ${A x} = {D s}$ is transformed to the objective function as a penalty term $M(Ax - Ds)^2$. 

The coefficients of the $D$ matrix are typically chosen to minimize the number of utilized slack variables. For example, if the inequality constraint is $x_1 + 2 x_2 + 6 x_3 + 10 x_4 \le 15$, the slack variables should cover the set $[0,15]$. For $b=15$, this is naively possible by using $b+1 = 16$ binary slack variables with ${D s} = 0 s_0 + 1 s_1 + 2 s_2 + 3 s_3 + ... + 14 s_{14} + 15 s_{15}$. This technique requires an additional constraint $s_0 + s_1 + s_2 + s_3 + ... + s_{14} + s_{15} = 1$ which enforces exactly one of the slack variables to be active. A more efficient technique presented in \cite{glover2002solving} and \cite{verma2020penalty} covers the slack range $[0,15]$ by using $4$ binary slack variables with ${D s} = s_1 + 2 s_2 + 4 s_3 + 8 s_4$. According to this method, $m$ slack variables cover the slack range corresponding to $2^0 + 2^1 + 2^2 + ... + 2^{m-1} = 2^m - 1$ units. In case $b$ is not of the form $2^m - 1$, we simply add a constant to both sides of the inequality. Using this method, we need $log(b+1)$ slack variables for a right hand side of $b$. Choosing the slack coefficients in this way reduces the number of slack variables from $O(b+1)$ (for the naive method) to $O(log(b+1))$.

The complexity of the augmented QUBO model is highly impacted by the introduction of slack variables, especially for larger values of $b$ because the size of the augmented $Q$ matrix increases from $(n,n)$ to $(n+m,n+m)$ where $n$ is the number of decision variables, and $m$ is the number of slack variables. Hence, the linear terms are increased by $O(m)$ and the quadratic interactions are increased by $O(nm+m^2)$. This dramatically increases the number of quadratic terms in the $Q$ matrix adversely impacting the performance of QUBO solver (\cite{verma2020penalty}). Moreover, the current generation of quantum annealers that rely on QUBO models have limitations on the size and number of quadratic interactions in the $Q$ matrix. In this paper, we demonstrate a linear transformation of the inequality constraint that reduces the number of slack variables and quadratic interaction terms. The findings are beneficial to the area of Noisy Intermediate-Scalar Quantum (NISQ) computing, which refers to combining classical and quantum computing approaches.

In summary, our goal is to find an effective way to handle a linear inequality in a QUBO model. Previous work on this includes \cite{glover2002solving} wherein the authors transform the Quadratic Knapsack Problem with a single inequality constraint into a QUBO by employing a quadratic infeasibility penalty via the use of slack variables. As discussed earlier, the authors covered a slack range of $63 (=2^m -1)$ units using $6 (=m)$ slack variables. More recently, the authors in \cite{verma2020penalty} presented a technique to handle slack variables for a single inequality constraint. Their method relied on partitioning the slack range and using an existing QUBO solver on multiple threads.

In comparison, we propose an approximation of the original problem using fewer slack variables and which is also amenable to partitioning using multiple threads. Our simple approach dramatically reduces the number of slack variables for generic inequality constraints of the form ${A x} \le b$. This reduction is dependent on a user-defined parameter $\rho$.  Note that while our proposed technique is valid for multiple linear inequalities (unlike \cite{verma2020penalty}), we will focus on a single linear inequality for the computational experiments. 

The paper is organized as follows. We present our reduced slack variables approach in Section 2. The details of the solution technique are presented in Section 3. Section 3 also examines the results on benchmark QUBO datasets followed by the conclusions in Section 4.

\section{Slack Variable Reduction Technique}
The base problem of interest is 
\begin{equation} \label{base}
min \; x'Qx \; s.t. \; Ax \leq b, \; x \in \{0,1\}.
\end{equation}
There are no positivity conditions on the $Q$ matrix. However, the coefficients of $A$ and $b$ are assumed to be positive integers and if not they are easily transformed via the substitution $y_i = 1-x_i$. As an example, the problem $min \; x'Qx \; s.t. \; 3 x_1 + 5 x_2 + 4 x_3 \leq 7$ would typically utilize a minimum of three slack variables. The slack coefficients given by $Ds = s_1 + 2 s_2 + 4 s_3$ covers the slack range $[0,7]$. The resulting equality constraint is transformed into the objective function through penalty term $M (Ax-Ds)^2$. The transformed objective function is $min \; x'Qx + M(3 x_1 + 5 x_2 + 4 x_3 -  s_1 - 2 s_2 - 4 s_3)^2$ and the degree of the resulting pseudo boolean polynomial is quadratic.

We now propose a scalar transformation of the inequality constraint that reduces the number of slack variables. The transformed problem is:
\begin{equation} \label{transform}
min \; x'Qx \; s.t. \; \frac{Ax}{\rho} \leq \frac{b}{\rho}, \; x \in \{0,1\}
\end{equation}
where $\rho$ is a user-defined positive integer. Note that $\rho=1$ corresponds to the base problem (\ref{base}). We simply divide both the sides of the linear inequality by $\rho$. Again, if $\frac{b}{\rho}$ is not of the form $2^m - 1$ we will add a constant to both sides. In this way, we now require only $O(log(\frac{b+1}{\rho}))$ slack variables thereby reducing the total number of slack variables by $O(log(\rho))$. For example, if we divide both sides of the inequality constraint $3 x_1 + 5 x_2 + 4 x_3 \leq 7$ by $7$, we get  $\frac{3}{7} x_1 + \frac{5}{7} x_2 + \frac{4}{7} x_3 \leq \frac{7}{7} = 0.42 x_1 + 0.71 x_2 + 0.57 x_3 \leq 1$. As a result, we arrive at the best case scenario with $\rho = b$ requiring only one additional slack variable.

Next, we interpret the constraint as $0.42 x_1 + 0.71 x_2 + 0.57 x_3 = \{0,1\}$. Further, we approximate the right-hand side by a slack variable $s_1$ leading to a total savings of two slack variables and ten new quadratic terms. Thus, the transformed problem becomes $min \; x'Qx + M'(0.42 x_1 + 0.71 x_2 + 0.57 x_3 - s_1)^2$. Note that the left-hand side of the fractional inequality could still take fractional values and this would not be exactly equal to the value corresponding to a binary slack variable ($0$ or $1$). In such a scenario, the penalty term would be non-zero even if the solution vector is feasible. This limitation of our proposed approximation can lead to a loss in solution quality, which we will analyze in the next section. 

A theoretical property of the proposed linear transformation is presented in the following lemma. It is important to understand the relationship between the optimal solution of the base problem (\ref{base}) and the transformed problem (\ref{transform}).
\begin{lemma}
	The optimal solution vector $x^*$ of (\ref{base}) and (\ref{transform}) will be the same if $\frac{Ax^*}{\rho}$ and $\frac{b}{\rho}$ are integers. 
\end{lemma}
\begin{proof}
	First, we will prove the lemma starting from (\ref{base}). If $x^*$ is the optimal solution for $min \; xQx \; s.t. \; Ax \leq b$, then $x^{*'}Qx^* \leq x'Qx \; \forall x \in X$ such that $Ax \leq b$. Assuming integer coefficients, we have  $Ax^* \in [0,b]$. Dividing both sides by a positive integer $\rho$, we get  $\frac{Ax^*}{\rho} \in  [0,\frac{b}{\rho}]$ if we assume that  $\frac{b}{\rho}$ is an integer. Thus, $x^*$ is feasible to the transformed problem (\ref{transform}). Since the objective function of (\ref{base}) and (\ref{transform}) are the same, $x^*$ is also optimal to the transformed problem (\ref{transform}). Thus, if $x^*$ is the optimal solution for (\ref{base}), it is also optimal for (\ref{transform}). Similarly, we can prove the other direction of this lemma.
\end{proof}

As a consequence of this lemma, we can guarantee to find the same optimal solution using the transformed problem as long as the optimal solution $x^*$ is such that $\frac{Ax^*}{\rho}$ and $\frac{b}{\rho}$ are integral. If $\frac{Ax^*}{\rho}$ is fractional, the resulting penalty term for the solution vector $x^*$ would be positive, given that $Ds$ would always be integral. This would disincentivize the solver to prefer $x^*$ in the neighborhood search of the transformed problem. Hence, there is no guarantee that the optimal solution for the base and the transformed problem would match in such a scenario.

Assuming integer coefficients, $Ax^*$ is a multiple of a prime number. Hence, if we enumerate $\rho$ over the set of all primes, we could guarantee to find the best solution. This could also be efficiently implemented using multiple threads. In Section 3, we evaluate the impact of different choices of $\rho$ on the computational results.

The transformation also requires adjustment to the original penalty coefficient $M$. As we divide both sides of the inequality of (\ref{base}) by $\rho$, the penalty coefficient $M$ should accordingly see a proportional increase. The penalty term of the base problem $M(Ax - Ds)^2$ is reduced by a factor of $\frac{1}{\rho^2}$ as we divide the inequality constraint by $\rho$. Therefore, if we use the same value of $M$ for both problems, the penalty term for the transformed problem will be smaller. Hence, we need to adjust the penalty coefficient $M'$ for the transformed problem. We use the relationship given by $M' = \rho^2 M$. In this way, we ensure that the penalty terms for the two problems are around the same orders of magnitude.

\section{Solution Technique and Results}

The QUBO solver used is a generic solver with no customization for handling constraints.   It is a tabu search based approach using the efficient 1-flip evaluation scheme detailed in (\cite{kochenberger2004unified}) and incorporates path relinking between elite solutions and backtracking to break out of local optima, along with strategic oscillation for neighborhood search intensification followed by restarts after these methods stop yielding improvements. The solution technique was implemented using DOCPLEX API in Python 3.6 and experiments were performed on a 3.40 GHz Intel Core i7 processor with 16 GB RAM running 64 bit Windows 7 OS.  The complete set of results can be found at \cite{slacks}.

We will present the impact of the proposed transformation on the benchmark Quadratic Knapsack Problem (QKP) instances described in \cite{billionnet2004exact}. The optimization problem of interest is $max \; x'Qx \; s.t. \; Ax <= b$. These datasets contain $[100,200,300]$ variables with varying density levels of $[25\%,50\%,75\%]$. A single linear inequality constraint is imposed on the quadratic binary objective. The coefficients of the left hand side of the constraint lie in the range $[0,100]$. We experiment with $\rho \in [1,10,100]$. Note that $\rho=1$ corresponds to the standard approach using the maximal number of slack variables. The number of slack variables reduces as $\rho$ increases leading to an approximation. We run our transformation approach detailed in Section 2 using the QUBO solver for a limited runtime of 100 seconds. We record the best objective function and best value of the left-hand side reported by the solver. All the values of $\rho$ yielded feasible solutions. Hence, the best left-hand side value never exceeded the right-hand side. The complete set of results can be obtained from \cite{slacks}. For the sake of brevity, we will present a summary of the results.

First, we will study the impact of $M$ on the computational experiments. Note that $M' = \rho^2 M$ is used for each instance. As $M$ increases, the penalty term associated with the violation of the linear inequality increases. Hence, the priority on the satisfaction of the linear inequality increases in comparison to the quadratic objective. Figure \ref{fig:fig0} presents a distribution of the number of best solutions found by the different combinations of $\rho$, $M$, and problem size. It is evident from the figure that $M=1000$ yielded $0$ or very few best solutions for the bigger problems in the limited timeframe. Hence, $M = 1000$ is not preferred over $M = 100$. We also note that $\rho = 10$ and $100$ lead to the best solution values in many instances. More specifically, $\rho = 1$ found the best solution in only $43$ out of $141$ total instances. Bigger values of $\rho$ lead to a more focused search wherein the left-hand side entries that are not a multiple of $\rho$ are heavily penalized. Hence, the solver can ignore such values. Further, these results demonstrate the utility of the transformation approach.

\begin{figure}[htbp]
	\centerline{\includegraphics[scale=0.9]{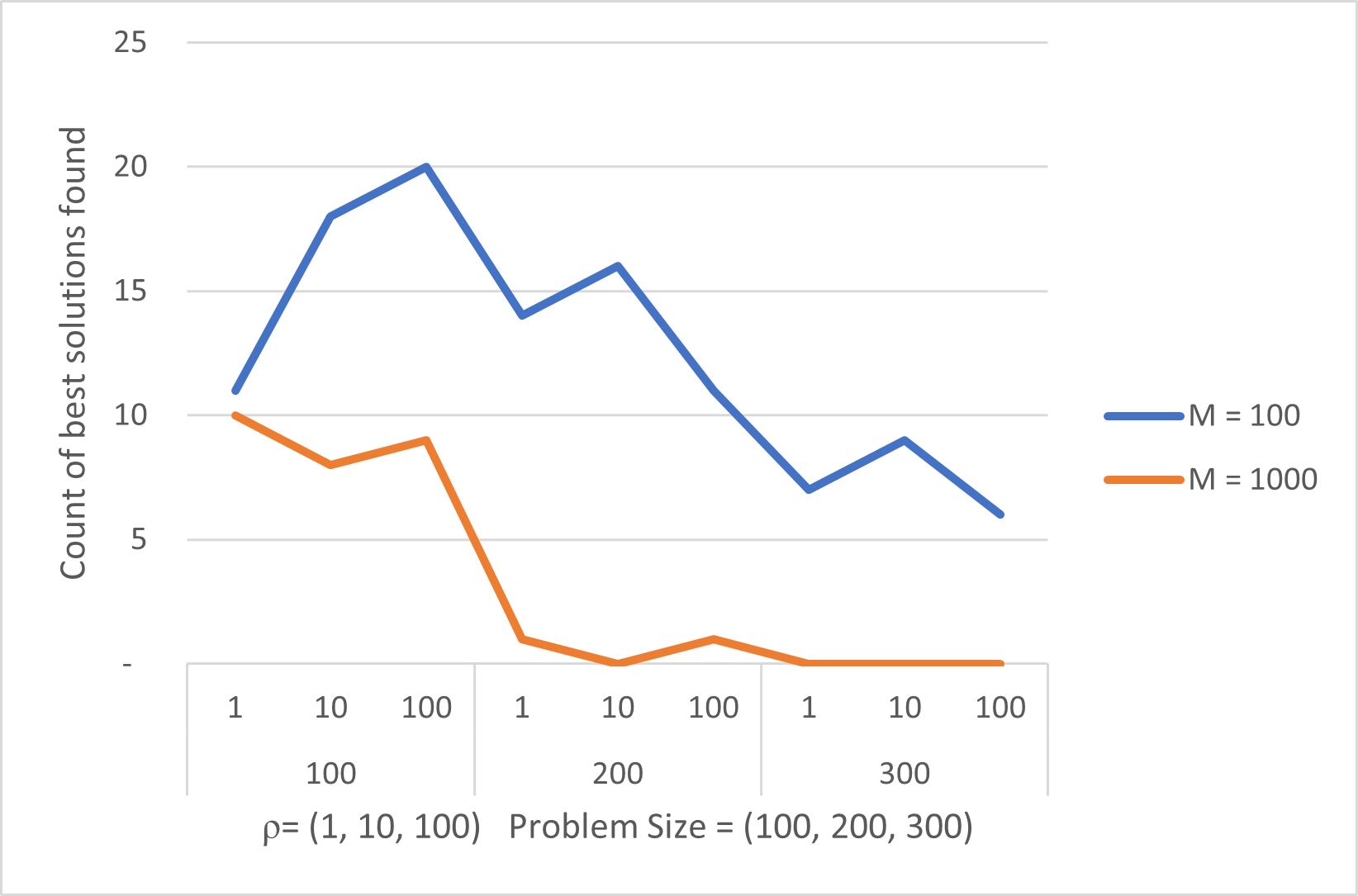}}
	\caption{Distribution of best solutions with respect to $M$ and problem size}
	\label{fig:fig0}
\end{figure}

Next, we explore the impact of $Q$ density on computational experiments and arrange the number of best solution values with respect to the density of the non-diagonal elements in the Q matrix and $\rho$ in Figure \ref{fig:fig01}. We cannot observe any discernible trends. In other words, the impact on $\rho$ on the best-found solution values does not change with respect to the density. 

\begin{figure}[htbp]
	\centerline{\includegraphics[scale=1.0]{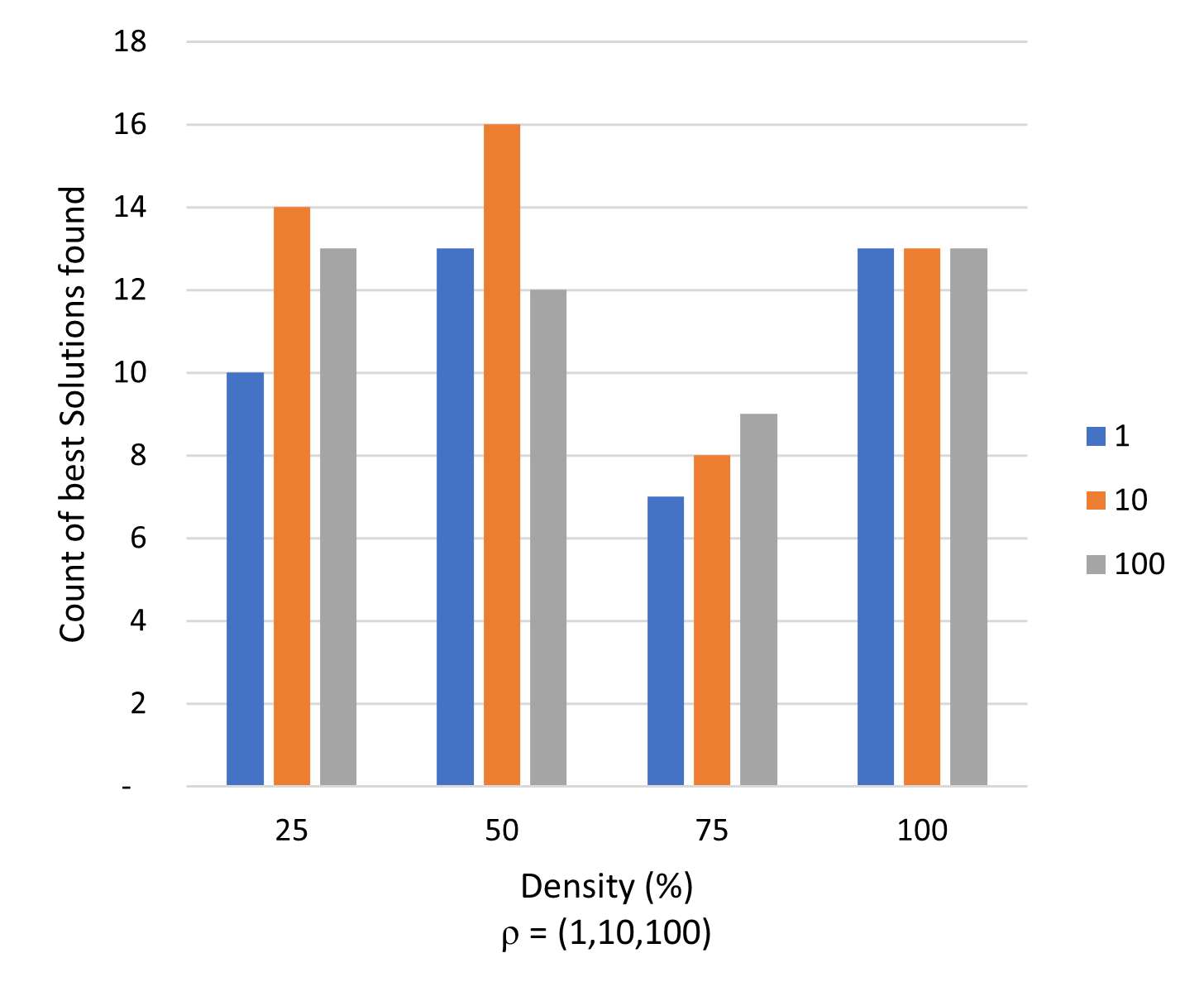}}
	\caption{Distribution of best solutions with respect to $Q$ density}
	\label{fig:fig01}
\end{figure}

We analyzed the effect of $Q$ density and $\rho$ on the variation between objective values on the problem set in Figure \ref{fig:fig02}. Hence, we measured the percent difference from the maximum value obtained for each problem as categorized by density and $\rho$. The percent deviation from the maximum provides a measure of variance in that if all problems in a $Q$ density set found the same maximum using different $\rho$ then the measure is 0\% and $\rho$ had no effect. The figure illustrates that the 75\% density problems showed the most variance regardless of $\rho$, but that $\rho = 1$ tended to allow for larger variances in objective value. This is consistent with the fact that $\rho = 1$ allows a larger solution landscape than the larger values of $\rho$. 

\begin{figure}[htbp]
	\centerline{\includegraphics[scale=0.35]{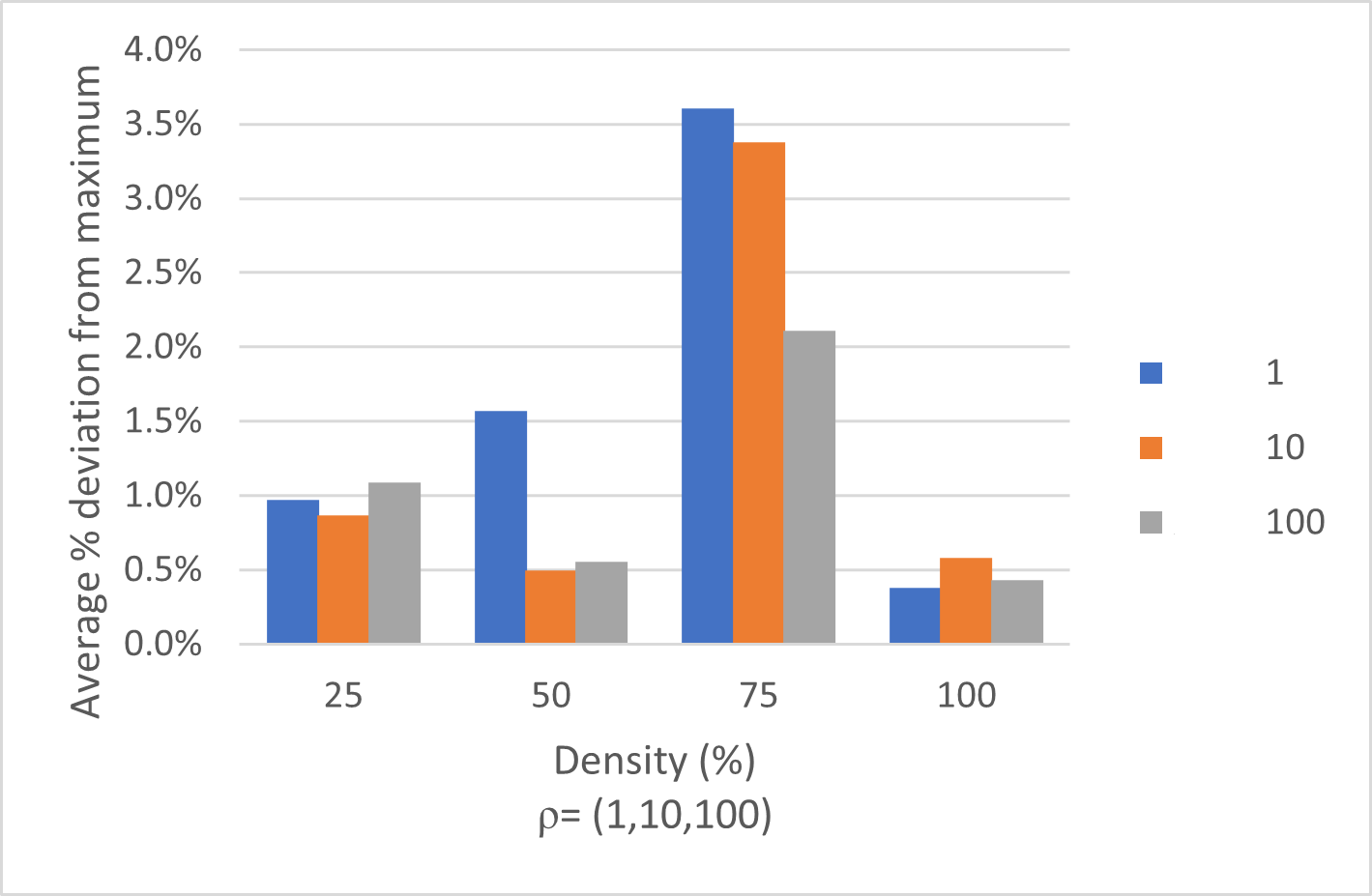}}
	\caption{Deviation of the best solution with respect to $Q$ density}
	\label{fig:fig02}
\end{figure}

For computational experiments on bigger datasets, we use the ten ORLIB QUBO instances presented in \cite{beasley1990or}. Similar to \cite{vyskovcil2019embedding} and \cite{verma2020penalty}, we introduce a single constraint $\sum_i x_i \le C$ where $C \in [20\%,50\%,80\%]$ of the number of variables $(2500)$. The problem of interest is $min \; x'Qx \; s.t. \sum_i x_i \le C$.

Figure \ref{fig:fig1} illustrates the progression of the average of the ten ORLIB problem objective values over time for a right-hand-side of $500$ and three different values of $\rho$.   Note that the scale of the y-axis is such that each mark is only $0.2\%$ change in objective value; hence the transformations produce results very close to the original problem ($\rho = 1$) but with fewer linear and quadratic terms.  

\begin{figure}[htbp]
	\centerline{\includegraphics[scale=0.44]{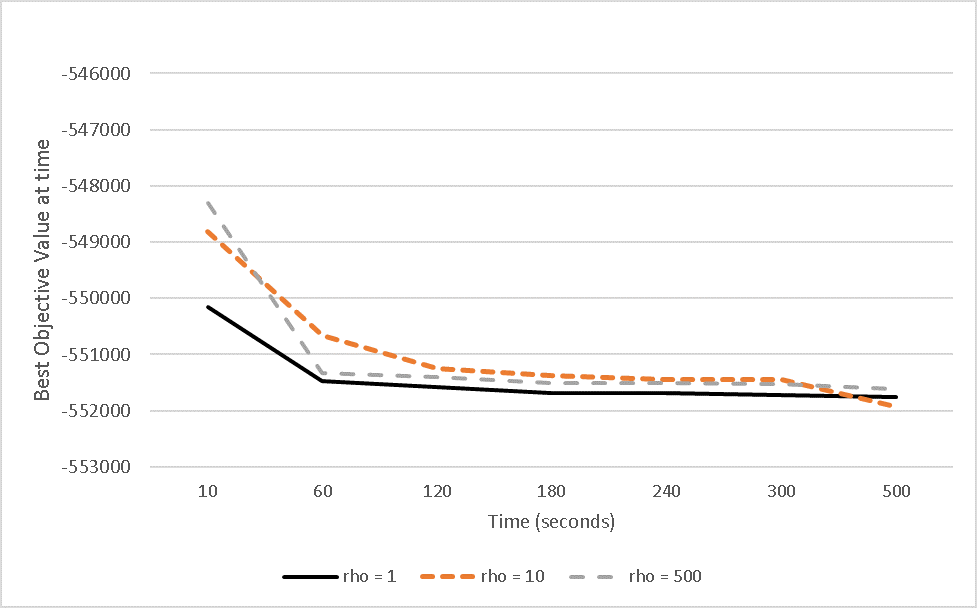}}
	\caption{Progression of the average of the objective values over time for $b = 500$}
	\label{fig:fig1}
\end{figure}

Figure \ref{fig:fig2} also uses the average objective values of the same ten ORLIB problems with right-hand-side equal to $50\% (1250)$.   The progression of solution values indicate that the transformed problems once again produce solutions that are consistently close to, or better than, the original formulation.  

\begin{figure}[htbp]
	\centerline{\includegraphics[scale=0.44]{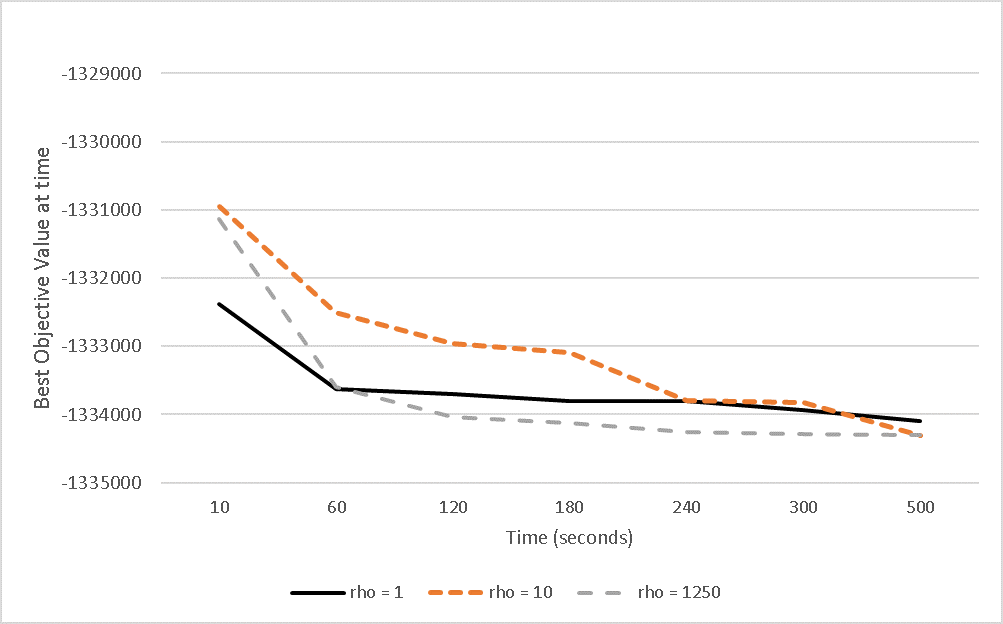}}
	\caption{Progression of the average of the objective values over time for $b = 1250$}
	\label{fig:fig2}
\end{figure}

Figure \ref{fig:fig3} illustrates the effect of the transformation when $\rho$ is set equal to a  right-hand-side of $2000$ which is greater than the value of the left-hand-side in the optimal solution. The left-hand-side at optimality for the ten problems is approximately $1650$ such that a slack of $350$ is needed to avoid a penalty.   However, since we use a single slack variable, that slack value is not available for $\rho = 2000$. Hence the transformed problem only reaches within $10\%$ of the best, while $\rho = 10$ performs much better because it can more closely approximate the needed slack value.

\begin{figure}[htbp]
	\centerline{\includegraphics[scale=0.70]{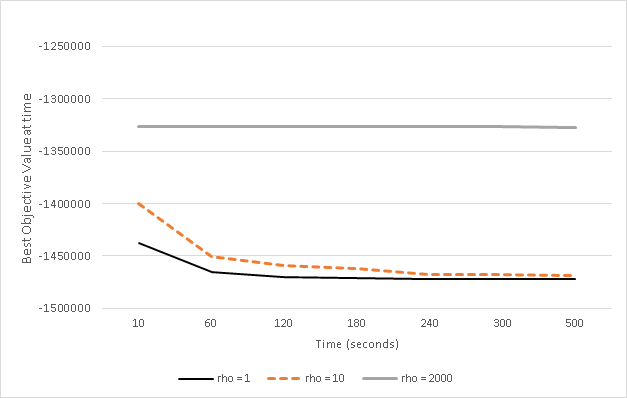}}
	\caption{Progression of the average of the objective values over time for $b = 2000$}
	\label{fig:fig3}
\end{figure}

\section{Conclusions}
We present a simple, yet powerful, slack variable and quadratic interaction term reduction technique for modeling inequality constraints in QUBO. The number of slack variables can be controlled based on a user-defined parameter and leads to comparable solutions for large QUBO instances. In terms of future research directions, we are adapting the current method to handle multiple constraints and investigating an approach for focusing the slack variable range.



\bibliographystyle{elsarticle-num} 
\bibliography{penaltyFunction}

\begin{thebibliography}{1}
\expandafter\ifx\csname url\endcsname\relax
  \def\url#1{\texttt{#1}}\fi
\expandafter\ifx\csname urlprefix\endcsname\relax\def\urlprefix{URL }\fi
\expandafter\ifx\csname href\endcsname\relax
  \def\href#1#2{#2} \def\path#1{#1}\fi

\bibitem{kochenberger2014unconstrained}
G.~Kochenberger, J.-K. Hao, F.~Glover, M.~Lewis, Z.~L{\"u}, H.~Wang, Y.~Wang,
  The unconstrained binary quadratic programming problem: a survey, Journal of
  Combinatorial Optimization 28~(1) (2014) 58--81.

\bibitem{glover2019quantum}
F.~Glover, G.~Kochenberger, Y.~Du, Quantum bridge analytics i: a tutorial on
  formulating and using qubo models, 4OR 17~(4) (2019) 335--371.

\bibitem{glover2002solving}
F.~Glover, G.~Kochenberger, B.~Alidaee, M.~Amini, Solving quadratic knapsack
  problems by reformulation and tabu search: Single constraint case, in:
  Combinatorial and global optimization, World Scientific, 2002, pp. 111--121.

\bibitem{verma2020penalty}
A.~Verma, M.~Lewis, Penalty and partitioning techniques to improve performance
  of qubo solvers, Discrete Optimization (2020) 100594.

\bibitem{kochenberger2004unified}
G.~A. Kochenberger, F.~Glover, B.~Alidaee, C.~Rego, A unified modeling and
  solution framework for combinatorial optimization problems, OR Spectrum
  26~(2) (2004) 237--250.

\bibitem{slacks}
Results, \url{https://github.com/amitverma1509/slacks}, accessed: 2020-09-01.

\bibitem{billionnet2004exact}
A.~Billionnet, {\'E}.~Soutif, An exact method based on lagrangian decomposition
  for the 0--1 quadratic knapsack problem, European Journal of operational
  research 157~(3) (2004) 565--575.

\bibitem{beasley1990or}
J.~E. Beasley, Or-library: distributing test problems by electronic mail,
  Journal of the operational research society 41~(11) (1990) 1069--1072.

\bibitem{vyskovcil2019embedding}
T.~Vyskocil, H.~Djidjev, Embedding equality constraints of optimization
  problems into a quantum annealer, Algorithms 12~(4) (2019) 77.

\end{thebibliography}

\end{document}